\documentclass[a4paper,10pt]{article}
\usepackage[utf8x]{inputenc}
\usepackage{amsmath,amsthm,amssymb,graphicx,nicefrac}

\title{An Algebraic Proof for the Identities for Degree of Syzygies in Numerical Semigroup}
\author{Neeraj Kumar\footnote{Dipartimento di Matematica, Universit\`a di Genova, Genova, Italy, \emph{kumar@dima.unige.it}}, 
Ivan Martino\footnote{Matematiska institutione, Stockholm Universitet, Stockholm, Sweden, \emph{martino@math.su.se}}}

\newtheorem{exa}{Example}[section]
\newtheorem{lemma}{Lemma}[section]
\newtheorem{thm}{Theorem}[section]
\newtheorem{rem}{Remark}[section]

\begin{document}

\maketitle

\begin{abstract}
In the article \cite{Fel_paper} two new identities for the degree of syzygies are given. We present an algebraic proof of them, using only basic homological algebra tools. We also extend these results.
\end{abstract}

\section{Introduction}
A numerical semigroup is a submonoid of $(\mathbb{N}, +)$ containing the zero element: it cames up in a lot of fields, for example the study of Diophantine equations, combinatorics and commutative algebra.

One of the most interesting and open question about numerical semigroups is 
about the Frobenius number (see \cite{Book_semigroup}). The question was posed by 
J. J. Sylvester, at the far 1884, in \cite{Syl} and it stays open until 
now: there does not exist an open formula for its Frobenius Number with respect to the generators of the semigroup. 
In fact, what we know is, such a formula is not a polynomial formula (see \cite{Curtis}).

In 2006, Leonid G. Fel solved this problem for the three generated semigroups. 
Even if his solution is not a proper formula, in the paper \cite{Fel_frobenius}, 
Fel showed how much the connection between the semigroups and other branch of mathematics 
is deep. 

The same author has also worked on the relation among the degree of syzygies (see \cite{Fel_paper}). 
The aim of this article is to give an algebraic proof of its main results. We are going to work in a non standard
grading, but similar computations in the standard graded case were done in \cite{Herzog}. 

The article has the following structure: in Section \ref{sec-fel}, we introduce some 
commutative algebra tools that we use throughout this article and we present the Fel results. 
Finally, in Section \ref{sec-proof}, we show our proof of the Fel results.

\section{Fel Results}\label{sec-fel}
In this section, we introduce some notations, main results given in \cite{Fel_paper} and 
some basic definition. 

We recall that a numerical semigroup is a submonoid of the natural number containing the zero 
element and that any numerical semigroup is finitely generated, i.e.
\[
  S=\left\langle d_1,\dots, d_m\right\rangle =\left\lbrace \sum_{i=1}^m a_id_i: a_i\in \mathbb{N}\right\rbrace;
\]
we call $d_1,\dots, d_m$ the generators of $S$. The basis is minimal if any generator 
cannot be expressed as a positive sum of the others. For convenience, 
we will always consider minimal basis and $d_1< d_2< \cdots <d_m.$

\paragraph{Notation} Fel uses to denote $S=\left\langle d_1,\dots, d_m\right\rangle$ with $S(\mathbf{d}^m)$. 
We will follow the literature style (expressing the generators) to be more clear.

It is known that $\gcd(d_1,\dots, d_m)=1$ is a necessary and sufficient condition for
$\mathbb{N}\setminus S$ being a finite set. The maximum of this set is called 
Frobenius number and is denoted by $g(S)$. The conductor of the semigroup $S$ is $c(S)=g(S)+1$ and it lies in $S$.

The Hilbert series $H(S; z)$ of the semigroup $S$ is the Hilbert series of its semigroup 
algebra $\mathbb{K}[[t^s: s\in S]]\subseteq \mathbb{K}[[t]]$. Thus, $H(S; z)=\sum_{s \in S} z^{s}$.
\begin{exa}\label{ex-2-3}
  Let $S=\left\langle 2, 3\right\rangle=\left\lbrace 0,2,3,4,\dots\right\rbrace\subset \mathbb{N}$. The Frobenius number is $1$ since $\mathbb{N}\setminus S=\left\lbrace 1 \right\rbrace$. Thus, $H(S, z)=1+z^2+z^3+\dots$ and so, using $\nicefrac{1}{(1-z)}=1+z+z^2+\dots$,  $H(S, z)=1+\nicefrac{z^2}{(1-z)}$; hence $H(S, z)=\nicefrac{(1-z+z^2)}{(1-z)}$.
\end{exa}
As we saw in the previous example, $H(S; z)$ is a rational function (see \cite{Stanley}). In particular, when we 
consider the \textit{grading} such as $deg(x_i)=d_i$ in $\mathbb{K}[x_1,\dots, x_n]$, then we have $H(\mathbb{K}[x_1,\dots, x_n]; z)=\nicefrac{1}{(1-t^{d_1})\cdots(1-t^{d_m})}$ and one has
\begin{equation}\label{eq-H-forms-1}
  H(S; z)=\frac{k(S; z)}{\prod_{i=1}^m(1-z^{d_i})}.
\end{equation}
where we call $k(S; z)$ the $k$-polynomial of the semigroup $S$.
Moreover, it is a fact that $\mathbb{K}[[t^s: s\in S]]$ is a geometrical object 
of dimension $1$, this implies that
\begin{equation}\label{eq-H-forms-2}
  H(S; z)=\frac{p(S; z)}{1-z},
\end{equation}
with $p(S;z)$ having the form (\ref{eq-p}), as proved soon.
The betti numbers of a numeri\-cal semigroup $S$ are the betti numbers of the 
semigroup algebra $\mathbb{K}[[t^s: s\in S]]$. We denote $\beta_i(S)$ the $i^{\text{th}}$ 
betti number. Similarly we define the graded betti numbers $\beta_{i,j}(S)$. 


\paragraph{Notation:} We are going to forget the $S$-dependence of $\beta_i(S)$ 
and $\beta_{i,j}(S)$.

It is well know that
\begin{equation}\label{eq-k}
  H(S; z)=\frac{\sum_{i=0}^{m-1} \sum_{j} (-1)^i \beta_{i,j} z^j}{\prod_{i=1}^m (1-z^{d_i})},
\end{equation}
where $j$ runs over the graded components of the $i^{\text{th}}$ homology; thus, using (\ref{eq-H-forms-1}) and (\ref{eq-H-forms-2}), 
\[
  \sum_{i=0}^{m-1} \sum_{j} (-1)^i \beta_{i,j} z^j =   \frac{\prod_{i=1}^m (1-z^{d_i})p(S; z)}{1-z}.
\]
One can rewrite it as
\begin{equation}\label{eq-key-formula}
  \sum_{i=0}^{m-1} \sum_{j}  (-1)^i \beta_{i,j} z^j = \prod_{i=1}^m (1+\dots+z^{d_i-1}) (1-z)^{m-1}p(S; z).
\end{equation}
The following lemma is a very useful in our proof.
\begin{lemma}\label{lem-usefull}
  Let $S$ be a numerical semigroup, then $p(S; 1)=1$.
\end{lemma}
\begin{proof}
We know that $H(S;z)=\sum_{s\in S} z^s $ and the existence of Frobenious 
number $g(S)$, together, imply that 
\[
   H(S; z) =\sum_{s\in S,\, s < g(S)} z^{s} + \sum_{i\in \mathbb{N}} z^{c(S)+i}.
\]
So, we can write it formally as
\[
   H(S; z) =\sum_{s\in S,\, s < g(S)} z^{s} + \frac{z^{c(S)}}{1-z};
\]
hence
\[
  H(S; z) =\frac{(1-z)\left( \sum_{s\in S,\, s < g(S)} z^{s}\right) + z^{c(s)}}{1-z}.
\]
Therefore one has 
\begin{equation}\label{eq-p}
 p(S; z)=(1-z)\left( \sum_{s\in S,\, s < g(S)} z^{s}\right) + z^{c(s)}.
\end{equation}
and, finally, $p(S; 1)=1$.
\end{proof}

We now present the Fel results, that we are going to prove. We always assume $H(S; z)$ Hilbert series is known in the form
\[
    H(S; z)=\frac{k(S; z)}{(1-z^{d_1})\cdots(1-z^{d_m})}.
\]
\begin{thm}\label{thm-not-complex}
Let $S=\left\langle d_1,\dots, d_m\right\rangle$. Then the following polynomial identities hold.
 \begin{equation}\label{eq-1}
      \sum_{i=0}^{m-1} \sum_{j}(-1)^i \beta_{i,j}j^r=0,
 \end{equation}
for $r=0,\dots, m-2$ and
\begin{equation}\label{eq-2}
      \sum_{i=0}^{m-1} \sum_{j} (-1)^i \beta_{i,j}j^{m-1}=(-1)^{m-1}(m-1)!\prod_{i=1}^m d_i.
\end{equation}
\end{thm}

Here we stress that the Fel result does not include the case $r=0$. We show in the next example how this identities work.

\begin{exa}
Let $S=\langle4,7,9\rangle$. One has that
\[
	H(S; z)=\frac{1-(z^{16}+z^{18}+z^{21})+(z^{25}+z^{30})}{(1-z^{4})(1-z^{7})(1-z^{9})}
\]
then $\beta_{0,0}=\beta_{1,16}=\beta_{1,18}=\beta_{1,21}=\beta_{2,25}=\beta_{2,30}=1$.
Thus for $r=0$, one has
\[
  \sum_{i=0}^{m-1} \sum_{j} (-1)^i \beta_{i,j}=-\beta_{1,16}-\beta_{1,18}-\beta_{1,21}+\beta_{2,25}+\beta_{2,30}+\beta_{0,0}=3-3=0;
\]
if $r=1$ then we have
\begin{eqnarray*}
  &&\sum_{i=0}^{m-1} \sum_{j} (-1)^i \beta_{i,j}j=\\
  &=&-\beta_{1,16}16-\beta_{1,18}18-\beta_{1,21}21+\beta_{2,25}25+\beta_{2,30}30+\beta_{0,0}0\\
  &=&-(16+18+21)+(25+30)=0;
\end{eqnarray*}
finally if $r=2$, we obtain
\begin{eqnarray*}
  &&\sum_{i=0}^{m-1} \sum_{j} (-1)^i \beta_{i,j}j^2=\\
  &=&-\beta_{1,16}16^2-\beta_{1,18}18^2-\beta_{1,21}21^2+\beta_{2,25}25^2+\beta_{2,30}30^2+\beta_{0,0}0^2\\
  &=&-(16^2+18^2+21^2)+(25^2+30^2)=(-1)^2\cdot2!\cdot4\cdot7\cdot9.
\end{eqnarray*}
\end{exa}

Moreover, it is possible to have a complex version of the previous identities. 
We define $w_q$ to be the number of generators divisible by  the integer number $q$.
%

\begin{thm}\label{thm-complex}
Let $S=\left\langle d_1,\dots, d_m\right\rangle$. For every integer numbers $q$ and $n$ such that $1\leq q \leq d_m$, $gcd(n,q)=1$ and $w_q>0$, then the following identity holds:
\begin{equation}\label{eq-3}
  \sum_{k=0}^{m-1} \sum_{j} (-1)^k \beta_{k,j}j^r e^{i\frac{2\pi n}{q} j}=0,
\end{equation}
for $r=0, 1,\dots, w_{q}-1$ (if $r=0$ we interprete $0^0$ as $1$). 
\end{thm}

In \cite{Fel_paper}, it is required $1\leq n \leq \frac{q}{2}$, too. But we found that this hypothesis can be removed.

\section{Proof of the results}\label{sec-proof}
In this section we will prove Theorem \ref{thm-not-complex} and Theorem 
\ref{thm-complex}. Before going on we remark and easy but useful fact.

\begin{rem}\label{rem-imp}
If $F(z)=(1-z)^k G(z)$, where $G(z)$ is a polynomial, then, for $i<k$, 
the $i^{\text{th}}$-derivative $F^{(i)}(1)=0$; moreover $F^{(k)}(1)=(-1)^k k!G(1)$
\end{rem}

\begin{proof}[Proof of Theorem \ref{thm-not-complex}]
We will prove (\ref{eq-1}) by induction. Let $r=0$; thus we substitute $z=1$ 
in (\ref{eq-key-formula}), and one has the equality promised.
Now, let $r=1$, thus we differentiate the l.h.s. of the equation (\ref{eq-key-formula}) with respect to $z$ and one has
\begin{eqnarray}
  \frac{d}{dz}\left[\sum_{i=0}^{m-1} \sum_{j} (-1)^i \beta_{i,j} z^j\right] 
  &=&\sum_{i=0}^{m-1} \sum_{j}  (-1)^i \beta_{i,j} \frac{d}{dz}[z^j] \nonumber\\
  &=&\sum_{i=0}^{m-1} \sum_{j}  (-1)^i \beta_{i,j} j z^{j-1}. \label{eq-diff-1-left}
\end{eqnarray}

Since the Remark \ref{rem-imp}, differentiating the r.h.s. of (\ref{eq-key-formula}) with respect to $z$ and, substituting $z=1$, we get all terms being zero. Hence, substituting $z=1$ in (\ref{eq-diff-1-left}), we get the l.h.s. of (\ref{eq-1}) and thus we proved the equality
\[
  \sum_{i=0}^{m-1} \sum_{j}  (-1)^i \beta_{i,j} j = 0.
\]

This is the basis of the induction and to prove the generic step for $k<m-1$ we claim 
that (\ref{eq-1}) is true for $r=1,\dots, k-1$. We differentiate $k$ times the l.h.s. of (\ref{eq-key-formula}) with respect to $z$:
\begin{eqnarray*}
  \frac{d^k}{dz^k}\left[\sum_{i=0}^{m-1} \sum_{j}  (-1)^i \beta_{i,j} z^j\right] 
  &=&\sum_{i=0}^{m-1} \sum_{j}  (-1)^i \beta_{i,j} \frac{d^k}{dz^k}[z^j]\\
  &=&\sum_{i=0}^{m-1} \sum_{j}  (-1)^i \beta_{i,j} j(j-1)\cdots (j-k) z^{j-k}.
\end{eqnarray*}
Since $j(j-1)\cdots (j-k)$ is a polynomial in $j$ and the leading term is $j^k$, we can 
split the sum in the following sums
\begin{equation*}
 \sum_{i=0}^{m-1} \sum_{j}  (-1)^i \beta_{i,j} j(j-1)\cdots (j-k) z^{j-k}=  \sum_{l=1}^{k}\left[\sum_{i=0}^{m-1} \sum_{j}  (-1)^i \beta_{i,j} c_l j^l z^{j-k}\right],
\end{equation*}
where $c_l$'s are integer numbers. Substituting $z=1$, one has that all the terms in the 
square brakets are zero for $l=1,\dots, k-1$ by induction and one obtains just the $l=k$ term:
\begin{equation*}
 \sum_{i=0}^{m-1} \sum_{j}  (-1)^i \beta_{i,j} j(j-1)\cdots (j-k)
  =\sum_{i=0}^{m-1} \sum_{j}  (-1)^i \beta_{i,j} j^k.
\end{equation*}
Differentiating the l.h.s. of (\ref{eq-key-formula}) with respect to $z$ (again, using the Remark \ref{rem-imp}), we have the same computation as before, this is zero for $z=1$. This ends the proof of the identity (\ref{eq-1}).

The only difference comes, when $r=m-1$. This is easy to prove with the same idea: 
differentiating $m-1$ times and substituting $z=1$, one has the l.h.s. of the (\ref{eq-2}). 
About the r.h.s. we need to consider the $m-1$-differentiation and use for the last time the Remark \ref{rem-imp}; we see that only 
one term survives. In fact, substituting $z=1$, one has
\[
  \frac{d^{m-1}}{dz^{m-1}}[r.h.s]=\left(\prod_{i=1}^m d_i\right) (-1)^{m-1}(m-1)!  p(S; 1).
\]
Since Lemma \ref{lem-usefull}, one obtains the claim:
\[
  \sum_{i=0}^{m-1} \sum_{j} (-1)^i \beta_{i,j} j^{m-1} = (-1)^{m-1}(m-1)! \prod_{i=1}^m d_i.
\]
\end{proof}

The idea of the proof is similar for the second theorem, but, now, we have more accurate substitution. In this proof we will also use a similar remark.

\begin{rem}\label{rem-imp-2}
Let $F(z)=h_1(z)h_2(z)\dots h_s(z) G(z)$, where $h_i(z)$ and $G(z)$ are a polynomials and let $\alpha\in \mathbb{C}$ be a root of $h_i(z)$ for exactly $k$ indices; then, for $i<k$, the $i^{\text{th}}$-differentiation of $F(z)$ is zero at $z=\alpha$.
\end{rem}

\begin{proof}[Proof of Theorem \ref{thm-complex}]
We want to prove the equality for $r=0$: in this case, we directly substitute $z=e^{ i \frac{2\pi n}{q}}$ in the expression (\ref{eq-key-formula}): l.h.s becomes
\begin{eqnarray*}
  \sum_{k=0}^{m-1}\sum_{j=0} (-1)^k\beta_{k,j} z^j
  &=&\sum_{k=0}^{m-1} \sum_{j}  (-1)^k \beta_{k,j} e^{i\frac{2\pi n}{q}j}\\
  &=&1+\sum_{k=1}^{m-1} \sum_{j} (-1)^k \beta_{i,j} e^{i\frac{2\pi n}{q}j}.
\end{eqnarray*}
We stress that in this case $j$ runs from $0$, because of $\beta_{0,0}=1$. If we show that r.h.s. of the (\ref{eq-key-formula}) is zero, under the substitution $e^{i\frac{2\pi n}{q}}$, then we have proved (\ref{eq-3}) for $r=0$. To show that, we need $w_q>0$. Under this condition, we observe that $e^{i\frac{2\pi n}{q}}$ is a $d_i^{\text{th}}$ root of unity for some $i$. Hence the r.h.s. of (\ref{eq-key-formula}) is zero for $z=e^{i\frac{2\pi n}{q}}$.

Now, let us prove the equality for $r=1$. As we did in the previous proof we use the (\ref{eq-key-formula}) and we differentiate it with respect to $z$. The l.h.s becames
\[
  \sum_{k=0}^{m-1} \sum_{j} (-1)^k \beta_{k,j}jz^{j-1}.
\]
Thus, we substitute $z=e^{i \frac{2\pi n}{q}\frac{j}{j-1}}$,
\begin{eqnarray*}
  \sum_{k=0}^{m-1} \sum_{j} (-1)^k \beta_{k,j}j z^{j-1}
  &=&\sum_{k=0}^{m-1} \sum_{j} (-1)^k \beta_{k,j}j e^{i \frac{2\pi n}{q}\frac{j}{j-1}(j-1)}\\
  &=&\sum_{k=0}^{m-1} \sum_{j} (-1)^k \beta_{k,j}j e^{i \frac{2\pi n}{q}j};
\end{eqnarray*}
this is the l.h.s of the (\ref{eq-3}). Now we differentiate the r.h.s. of (\ref{eq-key-formula}) with respect of $z$ and we substitute $z=e^{i \frac{2\pi n}{q}\frac{j}{j-1}}$), thus it vanishes. In fact, since $\gcd(n,q)=1$, $e^{i \frac{2\pi n}{q}\frac{j}{j-1}}$ is a $q^{\text{th}}$ root of unity. Moreover by hypotesis, $1\leq q \leq d_m$ and $r<w_q-1$, hence $e^{i \frac{2\pi n}{q}\frac{j}{j-1}}$ is a $d_i^{\text{th}}$ root of unity for exactly $w_q$ generators of the semigroup. Since $(1+z+\dots+z^{d_i-1})=\frac{1-z^{d_i}}{1-z}$, then $e^{i \frac{2\pi n}{q}\frac{j}{j-1}}$ annihilates $w_q$ of the factors $(1+z+\dots+z^{d_i-1})$; so we use the Remark \ref{rem-imp-2} and one has the claim.

Therefore we have proved that the statement for $r=1$. We want to gene\-ralize the statement 
for bigger $r$. In fact we are going to prove it by induction. The base of the induction is proved. 

Let the statement be true for $r=1,\dots, p<w_q-1$ and let us prove the result for $r=p+1$.

First of all we differentiate $r$ times the l.h.s. of (\ref{eq-key-formula}) with respect to $z$:
\begin{equation*}
 \frac{d^r}{dz^r}\left[\sum_{k=0}^{m-1} \sum_{j}  (-1)^k \beta_{k,j} z^j\right] 
=\sum_{k=0}^{m-1} \sum_{j} (-1)^k \beta_{k,j} j(j-1)\cdots (j-r) z^{j-r}.
\end{equation*}
So, we substitute $z= e^{i\frac{2\pi n}{q}\frac{j}{j-r}}$ and one has
\[
  \sum_{k=0}^{m-1} \sum_{j}  (-1)^k \beta_{k,j} j(j-1)\cdots (j-(r-1)) e^{i\frac{2\pi n}{q}j}.
\]
Therefore, expanding the product $j(j-1)\cdots (j-(r-1))$, one has $r$ summands of the form
\[
  \sum_{k=0}^{m-1} \sum_{j} (-1)^k \beta_{k,j} c_lj^l e^{i\frac{2\pi n}{q}j},
\]
where $c_l$'s are integer numbers, and by induction these summands are zero for $l=1,\dots, r-1=p$. 
Hence we obtain
\[
  \frac{d^r}{dz^r}\left[\sum_{k=0}^{m-1} \sum_{j} (-1)^k \beta_{k,j} z^j\right] 
 = \sum_{k=0}^{m-1} \sum_{j} (-1)^k \beta_{k,j} j^{r} e^{i\frac{2\pi n}{q}j},
\]
Now, let us consider $r^{\text{th}}$ differentiation of the r.h.s. of (\ref{eq-key-formula}). Similarly to the $r=1$ case, using the Remark \ref{rem-imp-2}; this is zero. Thus, we proved the equality. 
\end{proof}

\section*{Acknowledgements}
We thanks to Ralf Fr\"{o}berg, Mats Boij and Alexander Engstr\"{o}m for the support and the help given during the Pragmatic 2011.

\nocite{*}
\addcontentsline{toc}{section}{\textbf{Bibliography}}
\bibliographystyle{plain}

\end{document}